\documentclass[ejs,preprint]{imsart}

\usepackage{amsthm,amsmath}
\RequirePackage{natbib}
\RequirePackage[colorlinks,citecolor=blue,urlcolor=blue]{hyperref}
\RequirePackage[pdftex]{graphicx}
\doi{10.1214/154957804100000000}
\pubyear{0000}
\volume{0}
\firstpage{0}
\lastpage{0}

\startlocaldefs
\newcommand{\pteam}{\pi}
\newcommand{\Pteam}{\Pi}
\newcommand{\ptvec}{\boldsymbol{\pteam}}
\newcommand{\Ptvec}{\boldsymbol{\Pteam}}
\newcommand{\lteam}{\lambda}
\newcommand{\Lteam}{\Lambda}
\newcommand{\ltvec}{\boldsymbol{\lteam}}
\newcommand{\Ltvec}{\boldsymbol{\Lteam}}
\newcommand{\Zeta}{\mathrm{Z}}
\newcommand{\ppair}{\theta}
\newcommand{\Ppair}{\Theta}
\newcommand{\ppvec}{\boldsymbol{\ppair}}
\newcommand{\Ppvec}{\boldsymbol{\Ppair}}
\newcommand{\lpair}{\gamma}
\newcommand{\Lpair}{\Gamma}
\newcommand{\lpvec}{\boldsymbol{\lpair}}
\newcommand{\Lpvec}{\boldsymbol{\Lpair}}
\newcommand{\wwin}{w}

\newcommand{\nnum}{n}

\newcommand{\nnumvec}{\mathbf{\nnum}}

\newcommand{\win}{v}

\newcommand{\nt}{t}
\newcommand{\zerovec}{\mathbf{0}}

\newcommand{\mc}[1]{\mathcal{#1}}
\newcommand{\abs}[1]{\left\lvert#1\right\rvert}

\newcommand{\sref}[1]{Sec.~\ref{#1}}

\newcommand{\fref}[1]{Fig.~\ref{#1}}

\newcommand{\coord}{co\"{o}rdinate}

\numberwithin{equation}{section}
\theoremstyle{plain}
\newtheorem{lem}{Lemma}[section]
\newtheorem{defn}{Definition}[section]
\newtheorem{des}{Desideratum}[section]
\endlocaldefs
\begin{document}

\begin{frontmatter}

\title{Prior Distributions for the Bradley-Terry Model of Paired Comparisons}
\runtitle{Priors for Bradley-Terry}

\author{\fnms{John T.} \snm{Whelan}\ead[label=e1]{jtwsma@rit.edu}}
\address{School of Mathematical Sciences and\\
Center for Computational Relativity and Gravitation\\
Rochester Institute of Technology\\
Rochester, NY 14623}
\address{Max-Planck-Institut f\"{u}r Gravitationsphysik\\
(Albert-Einstein-Institut)\\
D-30167 Hannover, Germany\\
\printead{e1}
}

\runauthor{John T.\ Whelan}

\begin{abstract}
  The Bradley-Terry model assigns probabilities for the outcome of
  paired comparison experiments based on strength parameters
  associated with the objects being compared.  We consider different
  proposed choices of prior parameter distributions for Bayesian
  inference of the strength parameters based on the paired comparison
  results.  We evaluate them according to four desiderata motivated by
  the use of inferred Bradley-Terry parameters to rate teams on the
  basis of outcomes of a set of games: invariance under interchange of
  teams, invariance under interchange of winning and losing,
  normalizability and invariance under elimination of teams.  We
  consider various proposals which fail to satisfy one or more of
  these desiderata, and illustrate two proposals which satisfy them.
  Both are one-parameter independent distributions for the logarithms
  of the team strengths: 1) Gaussian and 2) Type III generalized
  logistic.
\end{abstract}



\end{frontmatter}

\section{Background}

\subsection{Paired Comparison Experiments and the Bradley-Terry Model}

A paired comparison experiment is a set of binary comparisons between
pairs out of a set of $\nt$ objects.  The Bradley-Terry model
\citep{Zermelo1929,BRADLEY01121952} assigns to each object $i$
($i=1,\ldots \nt$) a strength parameter $\pteam_i$, and defines
\begin{equation}
  \ppair_{ij} = \frac{\pteam_i}{\pteam_i+\pteam_j}
\end{equation}
as the probability that object $i$ will be preferred in any given
comparison with object $j$.  Note that $\ppair_{ji}=1-\ppair_{ij}$.
If $\nnum_{ij}$ the number of comparisons between $i$ and $j$, the
probability of any particular set of outcomes $D$ which includes
object $i$ being chosen over $j$ $\wwin_{ij}$ times is
\begin{equation}
  \label{e:likelihood}
  p(D|\{\ppair_{ij}\})
  = p(D|\{\pteam_i\})
  = \prod_{i=1}^{\nt} \prod_{j=i+1}^{\nt}
  \ppair_{ij}^{\wwin_{ij}} (1-\ppair_{ij})^{\nnum_{ij}-\wwin_{ij}}
\end{equation}
The model has been used in a number of contexts, ranging from taste
tests between different foods to games between chess players.  In the
context of the present paper, we are interested in sporting
competitions, so we will henceforth refer to the objects as ``teams''
and the comparisons as ``games''.  $\wwin_{ij}$ is thus the number of
games won by team $i$ against team $j$, and
$\nnum_{ij}=\wwin_{ij}+\wwin_{ji}$ is the number of games between
them.

\subsection{Bayesian Approach}

A typical problem is to make inferences about the strengths
$\{\pteam_i\}$, or equivalently the log-strengths $\{\lteam_i\}$,
given the results $D$.  Under a Bayesian approach, in terms of the
vector $\ptvec$ of team strengths $\{\pteam_i|i=1,\ldots,\nt\}$, the
posterior probability distribution will be, up to a
$\ptvec$-independent normalization constant,
\begin{equation}
  f_{\Ptvec|D}(\ptvec|D,I) \propto p(D|\{\pteam_i\}) f_{\Ptvec}(\ptvec|I)
\end{equation}
We are concerned with choices of the prior distribution
$f_{\Ptvec}(\ptvec|I)$, with a given choice represented symbolically
by $I$.

It is useful to define $\lteam_i=\ln\pteam_i$ and note that
\begin{equation}
  \ln\frac{\ppair_{ij}}{1-\ppair_{ij}} = \lteam_i-\lteam_j
  =: \lpair_{ij}
\end{equation}
and
\begin{equation}
  \ppair_{ij} = (1+e^{-\lpair_{ij}})^{-1},
  \qquad
  (1-\ppair_{ij}) = (1+e^{\lpair_{ij}})^{-1}
\end{equation}
Since the parameters are continuous, the probability density functions
transform as follows:
\begin{equation}
  \label{e:fltpt}
  f_{\Lteam_i}(\lteam_i) = e^{\lteam_i} f_{\Pteam_i}(e^{\lteam_i})
\end{equation}
\begin{equation}
  \label{e:fptlt}
  f_{\Pteam_i}(\pteam_i) = \frac{1}{\pteam_i} f_{\Lteam_i}(\ln\pteam_i)
\end{equation}
and likewise
\begin{equation}
  \label{e:flppp}
  f_{\Lpair_{ij}}(\lpair_{ij})
  = (1+e^{-\lpair_{ij}})^{-1}(1+e^{\lpair_{ij}})^{-1}
  f_{\Ppair_{ij}}([1+e^{-\lpair_{ij}}]^{-1})
\end{equation}
\begin{equation}
  \label{e:fpplp}
  f_{\Ppair_{ij}}(\ppair_{ij})
  = \ppair_{ij}^{-1}(1-\ppair_{ij})^{-1}
  \,f_{\Lpair_{ij}}(-\ln[\ppair_{ij}^{-1}-1])
\end{equation}

Note that the $\nt$ strengths $\{\pteam_i\}$ are only relevant in their
use to determine the probabilities $\{\ppair_{ij}\}$ (of which $\nt-1$
are independent), so we consider two probability distributions
$f_{\Ptvec}(\ptvec|I_1)$ and $f_{\Ptvec}(\ptvec|I_2)$ equivalent if
they produce the same marginalized distribution for the
$\{\ppair_{ij}\}$.
\begin{defn}
  Let $\lpvec$ represent a $\nt-1$-dimensional vector of linearly
  independent combinations of the log-odds-ratios $\{\lpair_{ij}\}$
  from which all $\nt(\nt-1)$ can be constructed according to
  \begin{equation}
    \label{e:lptbasis}
    \lpair_{ij} = \sum_{\alpha=1}^{\nt-1} C_{ij,\alpha} \lpair_{\alpha}
  \end{equation}
\end{defn}
Possible choices are
\begin{equation}
  \lpair_{12}, \lpair_{23}, \ldots, \lpair_{(\nt-1),\nt}
\end{equation}
or
\begin{equation}
  \lpair_{1\nt}, \lpair_{2\nt}, \ldots, \lpair_{(\nt-1),\nt}
\end{equation}
or
\begin{equation}
  \frac{1}{\sqrt{2}}\lpair_{12}, \frac{1}{\sqrt{6}}(\lpair_{13}+\lpair_{23}),
  \ldots,
  \frac{1}{\sqrt{\nt(\nt-1)}}[\lpair_{12}+\lpair_{23}
  +\ldots-(t-1)\lpair_{(\nt-1),\nt}]
\end{equation}
The advantage of working with the $\{\lpair_{ij}\}$ is that we need
not specify which basis we are using for $\lpvec$ because the Jacobian
determinants for transformations between different bases are constant.
\begin{defn}
  Two probability distributions are equivalent,
  $f_{\Ptvec}(\ptvec|I_1)\cong f_{\Ptvec}(\ptvec|I_2)$ (or
  $f_{\Ltvec}(\ltvec|I_1)\cong f_{\Ltvec}(\ltvec|I_2)$) if and only if
  $f_{\Lpvec}(\lpvec|I_1) = f_{\Lpvec}(\lpvec|I_2)$.
\end{defn}
\begin{lem}
  \label{l:scaling}
  A sufficient condition for $f_{\Ptvec}(\ptvec|I_1)\cong
  f_{\Ptvec}(\ptvec|I_2)$ is that there exists a scalar function
  $C(\ptvec)$ such that the transformation
  \begin{equation}
    \ptvec' = \ptvec C(\ptvec)
  \end{equation}
  converts the probability density $f_{\Ptvec}(\ptvec|I_1)$ into
  $f_{\Ptvec}(\ptvec|I_2)$, i.e.,
  \begin{equation}
    f_{\Ptvec'}(\ptvec'|I_1) = \frac{f_{\Ptvec}(\ptvec|I_1)}
    {
      \det\left\{\frac{\partial\pteam'_i}{\partial\pteam_j}\right\}
    }
    = f_{\Ptvec}(\ptvec'|I_2)
  \end{equation}
\end{lem}
\begin{proof}  The transformation leaves $\ppair_{ij}$ unchanged
  \begin{equation}
    \ppair'_{ij} = \frac{\pteam'_i}{\pteam'_i+\pteam'_j}
    = \frac{\pteam_iC(\ptvec)}{\pteam_iC(\ptvec)+\pteam_jC(\ptvec)}
    = \frac{\pteam_i}{\pteam_i+\pteam_j} = \ppair_{ij}
  \end{equation}
  and therefore $\lpair'_{ij}=\lpair_{ij}$ and the transformation
  $\lpvec\rightarrow\lpvec'$ leaves $f_{\Lpvec}(\lpvec|I)$ unchanged,
  and
  \begin{equation}
    f_{\Lpvec}(\lpvec|I_2)
    = f_{\Lpvec}(\lpvec'|I_2)
    = f_{\Lpvec'}(\lpvec'|I_1) = f_{\Lpvec}(\lpvec|I_1)
  \end{equation}
\end{proof}

\subsection{Motivation and Desiderata}

\label{s:motivation}

The primary interest motivating this work is design of rating systems
to evaluate teams based on the outcome of games between them.  To that
end, prior information which distinguishes between the teams is
inappropriate, as it would be considered ``unfair'' to build such
information into the rating system.  We are interested in rating
systems which obey as many as possibly of the following desiderata.
\begin{des}
  Invariance under interchange of teams. \label{d:interchange} A
  transformation $\ptvec\rightarrow\ptvec'$ which, for some $i$ and
  $j$, obeys $\pteam'_i=\pteam_j$, $\pteam'_j=\pteam_i$,
  $\pteam'_k=\pteam_k$ for all other $k$, should transform
  $f_{\Ptvec}(\ptvec|I_1)$ into an equivalent distribution
  $f_{\Ptvec}(\ptvec|I_2)\cong f_{\Ptvec}(\ptvec|I_1)$.
\end{des}
\begin{des}
  Invariance under interchange of winning and
  losing. \label{d:reflection} The transformation $\forall i:
  \pteam_i\rightarrow\pteam_i'=\frac{1}{\pteam_i}$, which corresponds
  to $\ltvec'=-\ltvec$, $\forall i,j:\ppair'_{ij}=1-\ppair_{ij}$, and
  $\lpvec'=-\lpvec$, should transform $f_{\Ptvec}(\ptvec|I_1)$ into an
  equivalent distribution $f_{\Ptvec}(\ptvec|I_2)\cong
  f_{\Ptvec}(\ptvec|I_1)$.  A distribution obeying this desideratum
  will satisfy $f_{\Lpvec}(\lpvec|I_1)=f_{\Lpvec}(-\lpvec|I_1)$.
\end{des}
\begin{des}
  Normalizability. \label{d:proper} $f_{\Lpvec}(\lpvec|I)$ should be
  a proper prior, which can be normalized to
  $\int_{-\infty}^{\infty}\cdots\int_{-\infty}^{\infty}
  d^{\nt-1}\lpair\,f_{\Lpvec}(\lpvec|I)=1$.
\end{des}
\begin{des}
  Invariant under elimination of teams. \label{d:elim} This
  desideratum assumes that a given principle can be used to generate
  prior distributions for any number of teams.  Let $\nt>2$, and
  define $\ptvec$ to be the vector of $\nt$ strengths, and $\ptvec'$
  to be the $(\nt-1)$-element vector with $\pteam'_i=\pteam_i$ for
  $i=0,\ldots,\nt-1$.  Suppose the principle generates priors
  $f_{\Ptvec'}(\ptvec'|I_{\nt-1})$ when there are $\nt-1$ teams and
  $f_{\Ptvec}(\ptvec|I_{\nt}) =
  f_{\Ptvec',\Pteam_{\nt}}(\ptvec',\pteam_{\nt}|I_{\nt})$ when there
  are $\nt$.  The prior associated with $I_{\nt-1}$ should be
  equivalent to that associated with $I_{\nt}$, marginalized over
  $\pteam_{\nt}$, i.e.
  \begin{equation}
    f_{\Ptvec'}(\ptvec'|I_{\nt-1})
    \cong
    \int_0^{\infty}d\pteam_{\nt}
    \,f_{\Ptvec',\Pteam_{\nt}}(\ptvec',\pteam_{\nt}|I_{\nt})
    = \int_0^{\infty}d\pteam_{\nt}
    \,f_{\Ptvec}(\ptvec|I_{\nt})
  \end{equation}
\end{des}

\subsection{Comparison via Prior Predictive Distribution}

A convenient way to quantify the effects of a prior, and thus to
compare different priors, is to construct the prior predictive
distribution
\begin{equation}
  \label{e:priorpred}
  p(D|\nnumvec,I)
  = \int_{0}^{\infty}\cdots \int_{0}^{\infty}
  d^{\nt}\pteam\,
  p(D|\ptvec,\nnumvec)\,f_{\Ptvec}(\ptvec|I)
  = \int_{-\infty}^{\infty}\cdots \int_{-\infty}^{\infty}
  d^{t-1}\lpair\, p(D|\lpvec,\nnumvec)\,f_{\Lpvec}(\lpvec|I)
\end{equation}
where the second equality holds because the $\nt-1$ log-rating
differences in $\lpvec$ determine the sampling distribution
$p(D|\ptvec,\nnumvec)$.
\begin{lem}
  \label{l:pred}
  For any $\nnumvec$, $p(D|\nnumvec,I)=p(D|\nnumvec,I')$ is a necessary
  condition for $f_{\Ptvec}(\ptvec|I)\cong f_{\Ptvec}(\ptvec|I')$.
\end{lem}
\begin{proof}
  Assume $f_{\Ptvec}(\ptvec|I)\cong f_{\Ptvec}(\ptvec|I')$.  Then, by
  definition $f_{\Lpvec}(\lpvec|I)= f_{\Lpvec}(\lpvec|I')$.  By
  \eqref{e:priorpred}, $p(D|\nnumvec,I)$ can be constructed from
  $f_{\Lpvec}(\lpvec|I)$ and $p(D|\lpvec,\nnumvec)$, and therefore
  $f_{\Lpvec}(\lpvec|I)= f_{\Lpvec}(\lpvec|I')$ implies
  $p(D|\nnumvec,I)=p(D|\nnumvec,I')$.
\end{proof}
We can use the prior predictive distribution to check the desiderata.
\begin{lem}
  \label{l:predinterchange}
  Given a $\nnumvec$ where $\nnum_{k\ell}=\nnum$ for all $k\ne\ell$,
  desideratum~\ref{d:interchange} implies
  $p(D|\nnumvec,I)=p(D'|\nnumvec,I)$ where $D$ and $D'$ differ only by
  the interchange of a pair of teams $i$ and $j$, i.e.,
  $\wwin'_{ij}=\wwin_{ji}$, $\wwin'_{ik}=\wwin_{jk}$, and
  $\wwin'_{jk}=\wwin_{ik}$, for all $k\not\in\{i,j\}$.
\end{lem}
\begin{proof}
  Defining $\nnum'_{k\ell}=\wwin'_{k\ell}+\wwin'_{\ell k}$ we see that
  for all $k\ne\ell$, $\nnum'_{k\ell}=\nnum=\nnum_{k\ell}$, i.e.,
  $\nnumvec'=\nnumvec$.  If we define $\ptvec\rightarrow\ptvec'$ as in
  the statement of desideratum~\ref{d:interchange}
  ($\pteam'_i=\pteam_j$, $\pteam'_j=\pteam_i$, $\pteam'_k=\pteam_k$
  for all $k\ne i,j$), we can see $p(D'|\lpvec',\nnumvec,I) =
  p(D'|\ptvec',\nnumvec,I) = p(D'|\ptvec',\nnumvec',I) = p(D|\ptvec,\nnumvec,I)
  = p(D|\lpvec,\nnumvec,I)$ where as usual
  $\lpair'_{ij}=\ln(\pteam'_i/\pteam'_j)$.  If
  desideratum~\ref{d:interchange} holds, we have
  $f_{\Lpvec'}(\lpvec'|I) = f_{\Lpvec}(\lpvec|I)$ and so
  \begin{equation}
    \begin{split}
      p(D'|\nnumvec,I)
      &= \int_{-\infty}^{\infty}\cdots \int_{-\infty}^{\infty}
      d^{t-1}\lpair'\, p(D'|\lpvec',\nnumvec)\,f_{\Lpvec}(\lpvec'|I)
      \\
      &= \int_{-\infty}^{\infty}\cdots \int_{-\infty}^{\infty}
      d^{t-1}\lpair'\, p(D|\lpvec,\nnumvec)\,f_{\Lpvec}(\lpvec|I)
      \\
      &= \int_{-\infty}^{\infty}\cdots \int_{-\infty}^{\infty}
      d^{t-1}\lpair\, p(D|\lpvec,\nnumvec)\,f_{\Lpvec}(\lpvec|I)
      = p(D|\nnumvec,I)
    \end{split}
  \end{equation}
  because the change of variables $\lpvec\rightarrow\lpvec'$ has unit
  Jacobian determinant and leaves the range of the integration
  variables unchanged.
\end{proof}
\begin{lem}
  For any $\nnumvec$, desideratum~\ref{d:reflection} implies
  $p(D|\nnumvec,I)=p(D'|\nnumvec,I)$ where
  $\wwin'_{ij}=\nnum_{ij}-\wwin_{ij}$.
\end{lem}
\begin{proof}
  Since $\wwin'_{ij}=\nnum_{ij}-\wwin_{ij}=\wwin_{ji}$, $\nnum'_{ij} =
  \wwin'_{ij}+\wwin'_{ji} = \wwin_{ji}+\wwin_{ij} = \nnum_{ij}$. The
  rest of the proof proceeds as with lemma~\ref{l:predinterchange},
  but with the appropriate definitions of $D\rightarrow D'$ and
  $\ptvec\rightarrow\ptvec'$.
\end{proof}
\begin{lem}
  For any $\nnumvec$, desideratum~\ref{d:proper} implies
  $p(D|\nnumvec,I)>0$ if $D$ is a set of results consistent with
  $\nnumvec$.
\end{lem}
\begin{proof}
  Since $p(D|\ppvec,I)>0$ for all $\ppvec$ with $0<\ppair_{ij}<1$, and
  $p(D|\lpvec,I)=p(D|\ppvec,I)$, we have $p(D|\lpvec,I)>0$ for all
  $\lpvec$ with $-\infty<\lpair_{ij}<\infty$.  Since
  $f_{\Lpvec}(\lpvec|I)\ge 0$ and
  $\int_{-\infty}^{\infty}\cdots\int_{-\infty}^{\infty}
  d^{\nt-1}\lpair\,f_{\Lpvec}(\lpvec|I)=1$, we must have
  \begin{equation}
    p(D|\nnumvec,I)
    = \int_{-\infty}^{\infty}\cdots \int_{-\infty}^{\infty}
    d^{t-1}\lpair\, p(D|\lpvec,\nnumvec)\,f_{\Lpvec}(\lpvec|I)
    \ge 0
  \end{equation}
\end{proof}
\begin{lem}\label{l:predelim}
  Given $\nt$ teams and a $\nnumvec$ with $\nnum_{i\nt}=0$,
  desideratum~\ref{d:elim} implies
  $p(D|\nnumvec,I_{\nt})=p(D|\nnumvec,I_{t-1})$ if $D$ is a set of results
  consistent with $\nnumvec$.
\end{lem}
\begin{proof}
  If $\nnum_{i\nt}=0$, $\pteam_{\nt}$ is irrelevant to the sampling
  distribution, and $p(D|\ptvec,\nnumvec)=p(D|\ptvec',\nnumvec)$ where
  $\ptvec'$ is the $(\nt-1)$-element vector with $\pteam'_i=\pteam_i$
  for $i=0,\ldots,\nt-1$, as in the statement of
  desideratum~\ref{d:elim}.  Thus
  \begin{equation}
    \begin{split}
      p(D|\nnumvec,I_{\nt}) &= \int_{0}^{\infty}\cdots \int_{0}^{\infty}
      d^{\nt}\pteam\,
      p(D|\ptvec,\nnumvec)\,f_{\Ptvec}(\ptvec|I_{\nt})
      \\
      &= \int_{0}^{\infty}\cdots \int_{0}^{\infty}
      d^{\nt-1}\pteam'\,
      p(D|\ptvec',\nnumvec)
      \int_{0}^{\infty} d\pteam_{\nt}\,f_{\Ptvec}(\ptvec|I_{\nt})
      \\
      &= \int_{0}^{\infty}\cdots \int_{0}^{\infty}
      d^{\nt-1}\pteam'\,
      p(D|\ptvec',\nnumvec) f_{\Ptvec'}(\ptvec'|I_{\nt})
    \end{split}
  \end{equation}
  Desideratum~\ref{d:elim} says that
  $f_{\Ptvec'}(\ptvec'|I_{\nt})\cong f_{\Ptvec'}(\ptvec'|I_{\nt-1})$,
  and lemma~\ref{l:pred} states that this implies
  $p(D|\nnumvec,I_{\nt})=p(D|\nnumvec,I_{\nt-1})$.
\end{proof}

\section{Choice of Prior Distribution}

\subsection{General Considerations for Special Cases}

\subsubsection{Two Teams}

\label{s:2teams}

When $\nt=2$, there is only one independent probability
$\ppair_{12}=\frac{\pteam_1}{\pteam_1+\pteam_2}$, so any distribution
$f_{\Ptvec}(\pteam_1,\pteam_2)$ reduces to a function
$f_{\Ppair_{12}}(\ppair_{12})$ via marginalization
\begin{equation}
  f_{\Lpair_{12}}(\lpair_{12})
  = \int_{-\infty}^{\infty} d\lteam_2\,f_{\Ltvec}(\lpair_{12}+\lteam_2,\lteam_2)
\end{equation}
where the transformation \eqref{e:fltpt} means
\begin{equation}
  f_{\Ltvec}(\lteam_1,\lteam_2)
  = e^{\lteam_1+\lteam_2} f_{\Ptvec}(e^{\lteam_1},e^{\lteam_2})
\end{equation}
and \eqref{e:fpplp} means
\begin{equation}
  f_{\Ppair_{12}}(\ppair_{12})
  = \ppair_{12}^{-1}(1-\ppair_{12})^{-1}\,f_{\Lpair_{12}}(-\ln[\ppair_{12}^{-1}-1])
\end{equation}
For the case of two teams, desiderata \ref{d:interchange} and
\ref{d:reflection} are equivalent, as both transformations reduce to
$\ppair_{12}\rightarrow 1-\ppair_{12}$, or equivalently
$\lpair_{12}\rightarrow -\lpair_{12}$.  They will be satisfied if and
only if $f_{\Lpair_{12}}(\lpair_{12})$ is an even function, or
equivalently if
$f_{\Ppair_{12}}(\ppair_{12})=f_{\Ppair_{12}}(1-\ppair_{12})$.
Desideratum~\ref{d:proper} will be satisfied if and only if
\begin{equation}
  \int_0^1 d\ppair_{12}\,f_{\Ppair_{12}}(\ppair_{12}) < \infty
\end{equation}
or equivalently
\begin{equation}
  \int_{-\infty}^{\infty} d\lpair_{12}\,f_{\Lpair_{12}}(\lpair_{12}) < \infty
\end{equation}

Suppose $f_{\Ppair_{12}}(\ppair_{12})$ belongs to the family of beta
distributions (which is the conjugate prior family for the likelihood
\eqref{e:likelihood}),
\begin{equation}
  \label{e:betadist}
  f_{\Ppair_{12}}(\ppair_{12})
  \propto \ppair_{12}^{\alpha-1}(1-\ppair_{12})^{\beta-1}
\end{equation}
then $f_{\Lpair_{12}}(\lpair_{12})$ is a generalized logistic
distribution of Type IV \citep{Prentice:1976,Nassar:2012}
\begin{equation}
  \label{e:betalog}
  f_{\Lpair_{12}}(\lpair_{12}) \propto
  (1+e^{-\lpair_{12}})^{-\alpha}(1+e^{\lpair_{12}})^{-\beta}
\end{equation}
Included in this family are
\begin{enumerate}
\item The Haldane prior \citep{Haldane1932,Jeffreys1939}
  $\alpha=\beta=0$, which is uniform $\lpair_{12}$.  This improper
  prior corresponds to ``total ignorance''.
\item The Jeffreys prior \citep{Jeffreys1939,Jeffreys:1946}
  $\alpha=\beta=\frac{1}{2}$.
\item The Bayes-Laplace prior $\alpha=\beta=1$, which is uniform in
  $\ppair_{12}$.  This is also the maximum entropy prior, if we assume
  a measure uniform in $\ppair_{12}$.
\end{enumerate}
For the beta family, desiderata \ref{d:interchange} and
\ref{d:reflection} will be satisfied if $\alpha=\beta$.  Desideratum
\ref{d:proper} will be satisfied if $\alpha,\beta>0$.

With $\nt=2$, the prior predictive probability for a set of results
which include $\wwin_{12}$ wins for team $1$ and
$\wwin_{21}=\nnum_{12}-\wwin_{12}$ wins for team $2$ will be
\begin{equation}
  p(D|\nnum_{12}) = \int_{0}^{1} d\ppair_{12}\,\ppair_{12}^{\wwin_{12}}
  (1-\ppair_{12})^{\wwin_{21}}\,f_{\Ppair_{12}}(\ppair_{12})
\end{equation}
If $f_{\Ppair_{12}}(\ppair_{12})$ is in the Beta family
\eqref{e:betadist}, it will be
\begin{equation}
  \label{e:predbeta}
  p(D|\nnum_{12},I_{\alpha,\beta})
  = \frac{\Gamma(\alpha+\beta)}{\Gamma(\alpha)\Gamma(\beta)}
  \frac{\Gamma(\alpha+\wwin_{12})\Gamma(\beta+\wwin_{21})}
  {\Gamma(\alpha+\beta+\nnum_{12})}
\end{equation}
In particular, for the Bayes-Laplace prior,
\begin{equation}
  p(D|\nnum_{12},I_{1,1})
  = \frac{\Gamma(1+\wwin_{12})\Gamma(1+\wwin_{21})}
  {\Gamma(2+\nnum_{12})}
  = \left((\nnum_{12}+1)\binom{\nnum_{12}}{\wwin_{12}}\right)^{-1}
\end{equation}
For the Jeffreys prior,
\begin{equation}
  \label{e:predjeff}
  p(D|\nnum_{12},I_{1/2,1/2})
  = \frac{\Gamma(\frac{1}{2}+\wwin_{12})\Gamma(\frac{1}{2}+\wwin_{21})}
  {\pi\Gamma(1+\nnum_{12})}
  = \frac{(2\wwin_{12}-1)!!(2\wwin_{21}-1)!!}{2^{\nnum_{12}}\nnum_{12}!}
\end{equation}
Viewing the Haldane prior as a limiting case,
$p(D|\nnum_{12},I_{1/2,1/2})=0$ unless $\wwin_{12}=\nnum_{12}$ or
$\wwin_{21}=\nnum_{12}$.  The prior predictive probabilities for
$\wwin_{12}=\nnum_{12}$ and $\wwin_{21}=\nnum_{12}$ depend on the
order in which the limits $\alpha\rightarrow 0$ and
$\beta\rightarrow 0$ are taken.

\subsubsection{Three Teams}

In the case $\nt=3$, the $\{\ppair_{ij}\}$ are related by
$\ppair_{ji}=1-\ppair_{ij}$ as well as
\begin{equation}
  \label{e:pconstraint}
  \ppair_{13}^{-1} - 1 = (\ppair_{12}^{-1} - 1)(\ppair_{23}^{-1} - 1)
\end{equation}
Although each $\ppair_{ij}$ is confined to the finite range $[0,1]$,
the surface defined by \eqref{e:pconstraint} is curved, which makes it
difficult to display the two-dimensional probability distribution
$f_{\Ppvec}(\ppvec)$ while preserving its intuitive interpretation.
On the other hand, in terms of the $\{\lpair_{ij}\}$ the constraints
are $\lpair_{ji}=-\lpair_{ij}$ and
\begin{equation}
  \lpair_{13} = \lpair_{12} + \lpair_{23}
\end{equation}
An especially convenient set of {\coord}s for displaying probability
distributions $f_{\Lpvec}(\lpvec)$ is
\begin{equation}
  \label{e:xycoords}
  x = \frac{1}{\sqrt{3}}(\lpair_{12}+\lpair_{13}),\qquad
  y = \lpair_{23}\ ,
\end{equation}
which can be inverted to give
\begin{equation}
  \label{e:gammaxy}
  \lpair_{12} = \frac{\sqrt{3}}{2}\,x - \frac{1}{2}\,y
  ,\qquad
  \lpair_{23} = y
  ,\qquad
  \lpair_{13} = \frac{\sqrt{3}}{2}\,x + \frac{1}{2}\,y
\end{equation}
These two-dimensional {\coord}s on the space of log-odds-ratios
$\lpvec$ are also two of the three {\coord}s on the space of strengths
$\ltvec$, according to
\begin{equation}
  \label{e:xylteam}
  x = \frac{1}{\sqrt{3}} (2\lteam_1 - \lteam_2 - \lteam_3)
  ,\qquad
  y = \lteam_2 - \lteam_3
  ,\qquad
  z = \sqrt{\frac{2}{3}}\, ( \lteam_1 + \lteam_2 + \lteam_3 )
\end{equation}
In these {\coord}s, to go from a distribution $f_{\Ltvec}(\ltvec)$ to
$f_{\Lpvec}(\lpvec)$ one simply converts $f_{\Ltvec}(\ltvec)$ into
$f_{XYZ}(x,y,z)$ (which involves a constant Jacobian determinant) and
then marginalizes over $z$.  An example of such a plot is shown in
\fref{f:t3ment}.
\begin{figure}[t!]
  \centering
  \includegraphics[width=0.8\columnwidth]{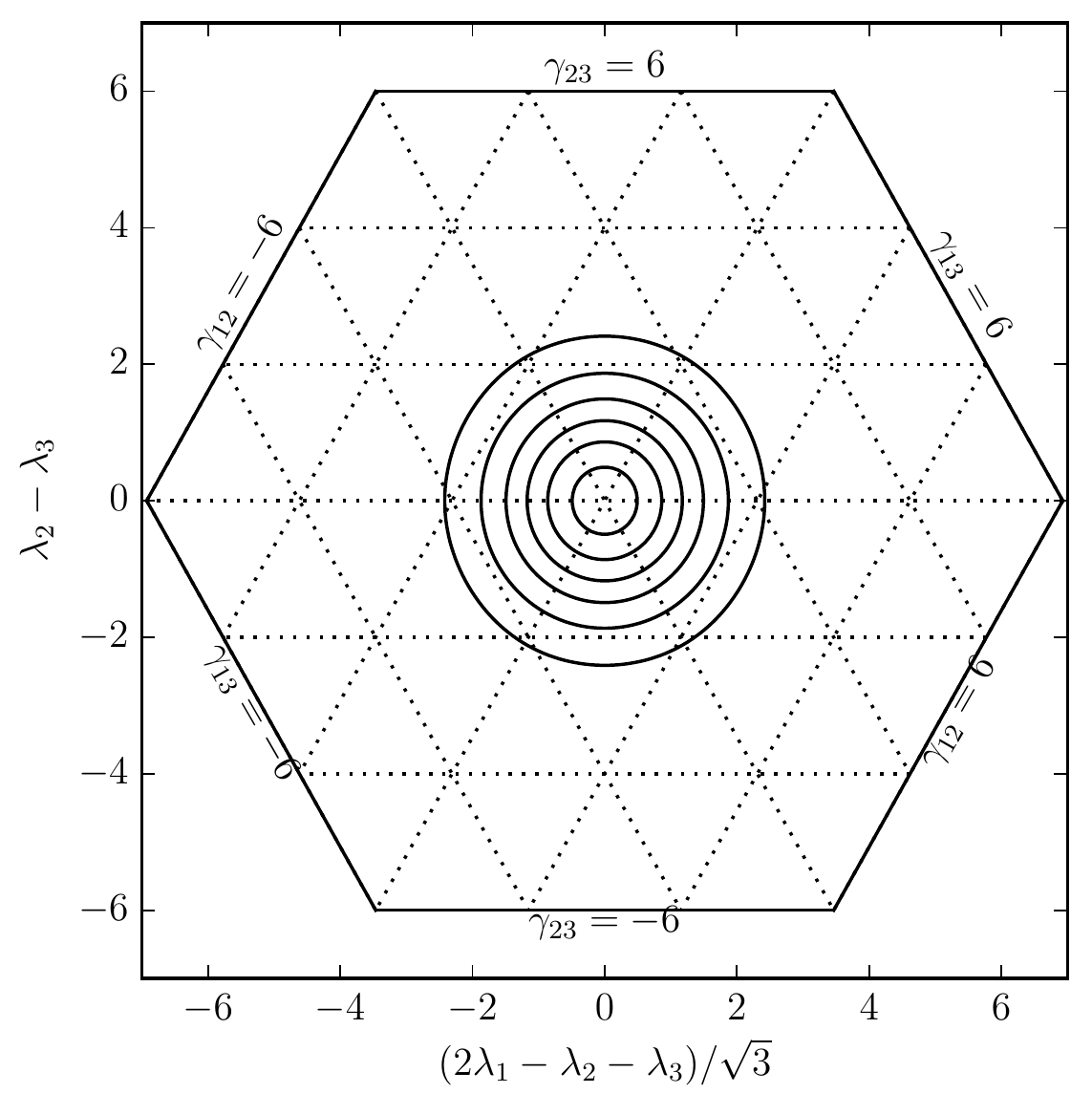}
  \caption{Contour plot of the prior probability distribution
    $f_{\Lpvec}(\lpvec|I_{\text{S}3})$ arising from the Maximum
    entropy prescription of \sref{s:maxent} when $\nt=3$.  The {\coord}s
    for the plot are the $x$ and $y$ defined in \eqref{e:xycoords},
    which determine all of the $\{\lpair_{ij}\}$ and thus the
    predicted probabilities.  The orthogonal direction
    $z=\sqrt{\frac{2}{3}}\, (\lteam_1 + \lteam_2 + \lteam_3)$ defined
    in \eqref{e:xylteam} is irrelevant to the predictions of the
    model.}
  \label{f:t3ment}
\end{figure}

\subsection{Evaluation of Prior Distributions}

We now consider several families of prior distributions which have
been proposed, and evaluate them according to the desiderata in
\sref{s:motivation}.

\subsubsection{Haldane Prior}

Perhaps the simplest prior that can be chosen is the improper prior
\begin{equation}
  f_{\Ltvec}(\ltvec|I_0)=\text{constant}\ ,
\end{equation}
uniform in all of the log-strengths, which is the generalization of
the Haldane prior considered in \sref{s:2teams}.  Then the
marginalized prior
probability distribution for the log-odds-ratios is
$f_{\Lpvec}(\lpvec|I_0)=\text{constant}$.  This prior obviously
satisfies desiderata \ref{d:interchange} and \ref{d:reflection}, as
well as desideratum~\ref{d:elim}.  Of course, since the prior
$f_{\Lpvec}(\lpvec|I_0)$ is improper, it violates desideratum
\ref{d:proper}.  The mode of the posterior
$f_{\Ltvec|D}(\ltvec|D,I_0)$ will be the maximum likelihood solution,
and the posterior will be normalizable under the conditions given by
\citet{Ford:1957} for the existence of the maximum likelihood
solution.

Note that while the improper prior
$f_{\Ptvec}(\ptvec|I_0')=\text{constant}$ produces a different prior
on the individual strengths, since
\begin{equation}
  f_{\Ltvec}(\ltvec|I_0') \propto \exp\left(\sum_{i=1}^{\nt} \lteam_i\right)
\end{equation}
the two are equivalent, $f_{\Ltvec}(\ltvec|I_0')\cong
f_{\Ltvec}(\ltvec|I_0)$, because $f_{\Ltvec}(\ltvec|I_0')$ depends
only on the sum of the log-strengths.

\subsubsection{Maximum Entropy}

\label{s:maxent}

The next prior of interest is one which maximizes the Shannon entropy.
For the case $\nt=2$, we defined the entropy as
\begin{equation}
  S_2 = - \int_0^1 d\ppair_{12}\,
  f_{\Ppair_{12}}(\ppair_{12}) \ln f_{\Ppair_{12}}(\ppair_{12})
\end{equation}
which made the maximum entropy prior
$f_{\Ppair_{12}}(\ppair_{12}|I_{\text{S}2})=\text{const}$.  In order for the
entropy of a continuous distribution $f(x)$ to be invariant, it needs
to be defined with a measure $\mu(x)$ which transforms as a density
under reparametrization.  Thus we can write
\begin{equation}
  S_2
  = - \int_0^1 d\ppair_{12}\,
  f_{\Ppair_{12}}(\ppair_{12})
  \ln \frac{f_{\Ppair_{12}}(\ppair_{12})}{\mu_{\Ppair_{12}}(\ppair_{12})}
  = - \int_0^1 d\lpair_{12}\,
  f_{\Lpair_{12}}(\lpair_{12})
  \ln \frac{f_{\Lpair_{12}}(\lpair_{12})}{\mu_{\Lpair_{12}}(\lpair_{12})}
\end{equation}
where we have assumed that
$\mu_{\Ppair_{12}}(\ppair_{12})=\text{constant}$ and thus, using
\eqref{e:flppp} to transform the measure,
\begin{equation}
  \mu_{\Lpair_{12}}(\lpair_{12}) \propto (1+e^{-\lpair_{12}})^{-1}(1+e^{\lpair_{12}})^{-1}
\end{equation}
If we maximize the entropy of a continuous distribution with
normalization as the only constraint, the probability density is
proportional to the measure, so
\begin{equation}
  f_{\Lpair_{12}}(\lpair_{12}|I_{\text{S}2})
  \propto (1+e^{-\lpair_{12}})^{-1}(1+e^{\lpair_{12}})^{-1}
\end{equation}
which is indeed of the form \eqref{e:betalog} with $\alpha=\beta=1$.

For general $\nt$, we could define by analogy a measure uniform in the
$\{\ppair_{ij}\}$, $\mu_{\Ppvec}(\ppvec)=\text{constant}$, and then
minimize the entropy
\begin{equation}
  S_{\nt}
  = - \int_0^1\cdots\int_0^1 d^{\nt(\nt-1)/2}\ppair\,
  f_{\Ppvec}(\ppvec)\ln f_{\Ppvec}(\ppvec)
\end{equation}
subject to the constraints that the probability density vanishes
unless the arguments satisfy
\begin{equation}
  \ppair^{-1}_{ij} = (\ppair^{-1}_{ik}-1)(\ppair^{-1}_{kj}-1)\qquad
  i=1,\ldots,\nt;\ j=i+1,\ldots,\nt; k=i+1,\ldots,j-1
\end{equation}
It is equivalent, and more straightforward, to confine the
distribution to the constraint surface by writing it in terms of the
$\nt-1$ unique $\lpair_\alpha$ parameters:
\begin{equation}
  S_{\nt}'
  = - \int_{-\infty}^{\infty}\cdots\int_{-\infty}^{\infty} d^{\nt-1}\lpair\,
  f_{\Lpvec}(\lpvec)\ln \frac{f_{\Lpvec}(\lpvec)}{\mu_{\Lpvec}(\lpvec)}
\end{equation}
where the measure is
\begin{equation}
  \label{e:meast}
  \mu_{\Lpvec}(\lpvec)
  \propto \prod_{i=1}^{\nt}\prod_{j=i+1}^{\nt}
  (1+e^{-\lpair_{ij}})^{-1}(1+e^{\lpair_{ij}})^{-1}
\end{equation}
As before, the maximum entropy distribution is
$f_{\Lpvec}(\lpvec|I_{\text{S}\nt})\propto\mu_{\Lpvec}(\lpvec)$, or
\begin{equation}
  \label{e:priorme}
  f_{\Lpvec}(\lpvec|I_{\text{S}\nt})
  \propto
  \prod_{i=1}^{\nt}\prod_{j=i+1}^{\nt}
  \left[
    1 + \exp \left(
      -\sum_{\alpha=1}^{\nt-1} C_{ij,\alpha} \lpair_\alpha
    \right)
  \right]^{-1}
  \left[
    1 + \exp \left(
      \sum_{\beta=1}^{\nt-1} C_{ij,\beta} \lpair_\beta
    \right)
  \right]^{-1}
\end{equation}
We can see from the form of \eqref{e:meast} that desiderata
\ref{d:interchange} and \ref{d:reflection} are satisfied.  It is also
easy to see that $f_{\Lpvec}(\lpvec|I_{\text{S}\nt})$ is exponentially
suppressed as any linear combination of the $\{\lpair_\alpha\}$ goes
to infinity, and therefore desideratum~\ref{d:proper}.  For example,
as $\lpair_\alpha\rightarrow\infty$,
\begin{equation}
  f_{\Lpvec}(\lpvec|I_{\text{S}\nt})
  \rightarrow \prod_{i=1}^{\nt}\prod_{j=i+1}^{\nt} e^{-\abs{C_{ij,\alpha}}\lpair_\alpha}
  = \exp \left(
    -\lpair_\alpha\sum_{i=1}^{\nt}\sum_{j=i+1}^{\nt} \abs{C_{ij,\alpha}}
  \right)
\end{equation}
However, we can see that desideratum~\ref{d:elim} is \emph{not}
satisfied by considering the case $\nt=3$ and showing that the marginal
distribution
\begin{equation}
  f_{\Lpair_{23}}(\lpair_{23}|I_{\text{S}3})\ne
  f_{\Lpair_{23}}(\lpair_{23}|I_{\text{S}2})
\end{equation}
Explicitly. using the {\coord}s \eqref{e:xycoords}, in which
$y=\lpair_{23}$,
\begin{equation}
  \begin{split}
    f_{\Lpvec}(x,y|I_{\text{S}3})
    &\propto
    (1+e^{\frac{\sqrt{3}}{2}x}e^{-y/2})^{-1}
    (1+e^{-\frac{\sqrt{3}}{2}x}e^{y/2})^{-1}
    (1+e^{-y})^{-1}(1+e^{y})^{-1}
    \\
    &\phantom{\propto}\times
    (1+e^{\frac{\sqrt{3}}{2}x}e^{y/2})^{-1}
    (1+e^{-\frac{\sqrt{3}}{2}x}e^{-y/2})^{-1}
    \ ,
  \end{split}
\end{equation}
which is plotted in \fref{f:t3ment}.
The marginalization integral can be done by partial fractions to give
\begin{equation}
  f_{\Lpair_{23}}(\lpair_{23}|I_{\text{S}3})
  = \int_{-\infty}^{\infty}f_{\Lpvec}(x,\lpair_{23}|I_{\text{S}3})\,dx
  \propto
  \frac{e^{\lpair_{23}}[2(1-e^{\lpair_{23}})+\lpair_{23}(1+e^{\lpair_{23}})]}
  {(e^{\lpair_{23}}-1)^3(1+e^{-\lpair_{23}})(1+e^{\lpair_{23}})}
\end{equation}
which is manifestly different from
\begin{equation}
  f_{\Lpair_{23}}(\lpair_{23}|I_{\text{S}2})
  \propto
  \frac{1}{(1+e^{-\lpair_{23}})(1+e^{\lpair_{23}})}
\end{equation}
by more than just a normalization constant.

\subsubsection{Jeffreys Prior}

\label{s:jeff}

The Jeffreys prior construction \citep{Jeffreys:1946} can be carried
out using the likelihood \eqref{e:likelihood}, to produce a prior
\begin{equation}
  f_{\Ltvec}(\ltvec|I_{\text{J}}) \propto \sqrt{\mc{I}(\ltvec)}
\end{equation}
where $\mc{I}(\ltvec)$ is the Fisher information associated with the
likelihood.  Since the likelihood is written in terms of the
$\{\ppair_{ij}\}$, or equivalently in terms of the $\{\lpair_{ij}\}$,
it is simpler to generate
\begin{equation}
  f_{\Lpvec}(\lpvec|I_{\text{J}}) \propto \sqrt{\mc{I}(\lpvec)}
\end{equation}
directly.  If we write the $\nt-1$ independent elements of $\lpvec$ as
\begin{equation}
  \tag{\ref{e:lptbasis}}
  \lpair_{ij} = \sum_{\alpha=1}^{\nt-1} C_{ij,\alpha} \lpair_{\alpha}
\end{equation}
we can write
\begin{equation}
  \mc{I}(\lpvec)
  =
  -E\left[
    \det\left\{
      \frac{\partial^2\ell(\lpvec;D)}
      {\partial\lpair_{\alpha}\partial\lpair_{\beta}}
    \right\}
  \right]
\end{equation}
where
\begin{equation}
  \ell(\lpvec;D) = \ln p_{D|\lpvec}(D|\lpvec)
  = \sum_{i=1}^{\nt}\sum_{j=i+1}^{\nt}
  (\wwin_{ij}\lpair_{ij}-\nnum_{ij}\ln[1+e^{\lpair_{ij}}])
\end{equation}
The linear form of \eqref{e:lptbasis} allows us to determine the
Fisher information matrix for the $\{\lpair_\alpha\}$ from that for
the $\{\lpair_{ij}\}$ as
\begin{equation}
  \frac{\partial^2\ell(\lpvec;D)}
  {\partial\lpair_{\alpha}\partial\lpair_{\beta}}
  =
  \sum_{i=1}^{\nt}\sum_{j=i+1}^{\nt}
  \sum_{i'=1}^{\nt}\sum_{j'=i'+1}^{\nt}
  C_{ij,\alpha}
  C_{i'j',\beta}
  \frac{\partial^2\ell(\lpvec;D)}
  {\partial\lpair_{ij}\partial\lpair_{i'j'}}
\end{equation}
Since the log-likelihood is relatively simple written in terms of the
$\{\lpair_{ij}\}$, we can write
\begin{equation}
  \frac{\partial\ell(\lpvec;D)}
  {\partial\lpair_{ij}}
  = \wwin_{ij} - \nnum_{ij}e^{\lpair_{ij}}(1+e^{\lpair_{ij}})^{-1}
  = \wwin_{ij} - \nnum_{ij}(1+e^{-\lpair_{ij}})^{-1}
\end{equation}
and
\begin{equation}
  \frac{\partial^2\ell(\lpvec;D)}
  {\partial\lpair_{ij}\partial\lpair_{i'j'}}
  = - \delta_{ii'}\delta_{jj'}
  \nnum_{ij}e^{-\lpair_{ij}}(1+e^{-\lpair_{ij}})^{-2}
  = - \delta_{ii'}\delta_{jj'}
  \nnum_{ij}(1+e^{-\lpair_{ij}})^{-1}(1+e^{\lpair_{ij}})^{-1}
\end{equation}
so
\begin{equation}
  \label{e:fisheralpha}
  - \frac{\partial^2\ell(\lpvec;D)}
  {\partial\lpair_{\alpha}\partial\lpair_{\beta}}
  =
  \sum_{i=1}^{\nt}\sum_{j=i+1}^{\nt}
  C_{ij,\alpha}
  C_{ij,\beta}
  \nnum_{ij}(1+e^{-\lpair_{ij}})^{-1}(1+e^{\lpair_{ij}})^{-1}
\end{equation}
We can see by inspection of \eqref{e:fisheralpha} that the Jeffreys
prior always satisfies desideratum~\ref{d:reflection}.  We can verify
that in the case $\nt=2$, for which there is only one independent
$\lpair_{\alpha}$, the Jeffreys prior becomes
\begin{equation}
  f_{\Lpvec}(\lpair_{12}|I_{\text{J}2})
  \propto
  (1+e^{-\lpair_{ij}})^{-1/2}(1+e^{\lpair_{ij}})^{-1/2}
\end{equation}
which is of the form \eqref{e:betalog} with $\alpha=\beta=\frac{1}{2}$
as before.

If we write
\begin{equation}
  - \frac{\partial^2\ell(\lpvec;D)}
  {\partial\lpair_{\alpha}\partial\lpair_{\beta}}
  =
  \sum_{i=1}^{\nt}\sum_{j=i+1}^{\nt}
  C_{ij,\alpha}
  C_{ij,\beta}
  \nnum_{ij}(e^{\lpair_{ij}/2}+e^{-\lpair_{ij}/2})^{-2}
\end{equation}

Note that if we make the specific choice
$\lpair_{\alpha}=\lpair_{\alpha,\alpha+1}$, we have
\begin{equation}
  \label{e:gammachoice}
  \lpair_{ij} = \sum_{\alpha=i}^{j-1} \lpair_\alpha
\end{equation}
which makes
\begin{equation}
  \label{e:Cijdiff}
  C_{ij,\alpha} =
  \begin{cases}
    1 & i \le \alpha \le j-1\\
    0 & \text{otherwise}
  \end{cases}
\end{equation}
and then the Fisher information matrix \eqref{e:fisheralpha} is
\begin{equation}
  - \frac{\partial^2\ell(\lpvec;D)}
  {\partial\lpair_{\alpha}\partial\lpair_{\beta}}
  =
  \sum_{\substack{i,j\\ i \le \alpha \le j-1 \\ i \le \beta \le j-1}}
  \nnum_{ij}(e^{\lpair_{ij}/2}+e^{-\lpair_{ij}/2})^{-2}
\end{equation}

For $\nt>2$, the Fisher information matrix \eqref{e:fisheralpha} depends
on the number of games $\nnum_{ij}$ to be played between each pair of
teams.  However, in order to satisfy desideratum~\ref{d:interchange},
we need to have the same $\nnum_{ij}$ for each pair of teams, in which
case this $\nnum_{ij}$ becomes a constant which can be absorbed into the
normalization, and the prescription becomes unique.

By explicitly examining $\nt=3$, we will show that the Jeffreys prior
with all $\{\nnum_{ij}\}$ equal fails to satisfy desideratum
\ref{d:elim}.\footnote{It can be seen to satisfy \ref{d:proper} by
  noting that as a linear combination of the $\{\lpair_\alpha\}$ goes
  to infinity, at most one of the $\{m_{ij}\}$ can remain finite; the
  other two will be exponentially suppressed, and thus each term in
  the square root in \eqref{e:priorjeff3} will go to zero exponentially.}
In the case $\nt=3$, the vector of independent
log-odds-ratios $\lpvec$ is two-dimensional, and the elements of the
Fisher information matrix are
\begin{subequations}
  \begin{gather}
    - \frac{\partial^2\ell(\lpvec;D)}
    {\partial\lpair_{1}\partial\lpair_{1}}
    =
    M_{12} + M_{13}
    \\
    - \frac{\partial^2\ell(\lpvec;D)}
    {\partial\lpair_{1}\partial\lpair_{2}}
    =
    M_{13}
    \\
    - \frac{\partial^2\ell(\lpvec;D)}
    {\partial\lpair_{2}\partial\lpair_{2}}
    =
    M_{13} + M_{23}
  \end{gather}
\end{subequations}
where
\begin{equation}
  M_{ij} := \nnum_{ij}(e^{\lpair_{ij}/2}+e^{-\lpair_{ij}/2})^{-2}
  =: \nnum_{ij} m_{ij}
\end{equation}
The Fisher information is the determinant of this matrix
\begin{equation}
  \mc{I}(\lpvec) = (M_{12}+M_{13})(M_{13}+M_{23})-M_{13}^2
  = M_{12}M_{13}+M_{12}M_{23}+M_{13}M_{23}
\end{equation}
so the Jeffreys prior is
\begin{equation}
  \label{e:priorjeff3}
  f_{\Lpvec}(\lpvec|I_{\text{J}3})
  \propto
  \sqrt{M_{12}M_{13}+M_{12}M_{23}+M_{13}M_{23}}
  \propto
  \sqrt{m_{12}m_{13}+m_{12}m_{23}+m_{13}m_{23}}
\end{equation}
We can show that the Jeffreys prior fails to satisfy
desideratum~\ref{d:elim} by using the posterior predictive
distribution and Lemma~\ref{l:predelim}.  Suppose $\nnum_{12}=2$ and
$\nnum_{i3}=0$.  Then \eqref{e:predjeff} implies
\begin{equation}
  p(D|\nnumvec,I_{\text{J}2}) =
  \begin{cases}
    0.125 & \wwin{12}=1 \\
    0.375 & \wwin{12}=0 \hbox{ or } 2
  \end{cases}
\end{equation}
We can evaluate $p(D|\nnumvec,I_{\text{J}3})$ numerically, and find
\begin{equation}
  p(D|\nnumvec,I_{\text{J}3}) \approx
  \begin{cases}
    0.108 & \wwin{12}=1 \\
    0.392 & \wwin{12}=0 \hbox{ or } 2
  \end{cases}
\end{equation}
showing explicitly that
$p(D|\nnumvec,I_{\text{J}3})\ne p(D|\nnumvec,I_{\text{J}2})$ and
therefore desideratum~\ref{d:elim} is violated.

\subsubsection{Dirichlet Distribution}
\label{s:dirichlet}

\citet{Chen19849} discuss Bayesian estimators for the Bradley-Terry
model starting with a Dirichlet distribution
\begin{equation}
  \label{e:dirichlet}
  f_{\Ptvec}(\ptvec|I_{\text{D}\nt})
  = \frac{\Gamma\left(\sum_{i=1}^{\nt} \alpha_i\right)}
  {\prod_{i=1}^{\nt} \Gamma(\alpha_i)}
  \left(\prod_{i=1}^{\nt} \pteam_i^{\alpha_i-1}\right)
  \delta\left(1-\sum_{i=1}^{\nt} \pteam_i\right)
\end{equation}
where $\delta(x)$ is the Dirac delta function.  In particular, they
note that the marginal distribution for any of the $\{\ppair_{ij}\}$
is a beta distribution with parameters $\alpha=\alpha_i$ and
$\beta=\alpha_j$:
\begin{equation}
  \label{e:dirichletmarg}
  f_{\Ppair_{ij}}(\ppair_{ij}|I_{\text{D}\nt})
  = \frac{\Gamma(\alpha_i+\alpha_j)}{\Gamma(\alpha_i)\Gamma(\alpha_j)}
  \ppair_{ij}^{\alpha_i-1}(1-\ppair_{ij})^{\alpha_j-1}
\end{equation}
The Dirichlet prior satisfies desideratum~\ref{d:proper} as long as
all of the $\{\alpha_i\}$ are positive; it also satisfies desideratum
\ref{d:interchange} if all of the parameters $\{\alpha_i\}$ are equal
to a the same value $\alpha$.  Although the delta function enforcing
the constraint $\sum_{i=1}^{\nt} \pteam_i^{\alpha_i}=1$ means that the
different $\{\pteam_i\}$ are not independently distributed under
$I_{\text{D}\nt}$, we can see that desideratum~\ref{d:elim} is
satisfied by defining a change of variables
\begin{equation}
  \pteam'_i = \frac{\pteam_i}{1-\pteam_{\nt}} \qquad i = 1,\ldots,\nt-1
\end{equation}
under which the probability density \eqref{e:dirichlet} becomes
\begin{equation}
  f_{\Ptvec',\Pteam_{\nt}}(\ptvec',\pteam_{\nt}|I_{\text{D}\nt})
  = \frac{\Gamma\left(\sum_{i=1}^{\nt} \alpha_i\right)}
  {\prod_{i=1}^{\nt} \Gamma(\alpha_i)}
  \left(\prod_{i=1}^{\nt-1} {\pteam'_i}^{\alpha_i-1}\right)
  \delta\left(1-\sum_{i=1}^{\nt-1} \pteam'_i\right)
  (1-\pteam_{\nt})^{\sum_{i=1}^{\nt-1}\alpha_i-1}{\pteam_{\nt}}^{\alpha_{\nt}-1}
\end{equation}
which, when we marginalize over $\pteam_{\nt}$, gives
\begin{equation}
  f_{\Ptvec'}(\ptvec'|I_{\text{D}\nt})
  = \frac{\Gamma\left(\sum_{i=1}^{\nt-1} \alpha_i\right)}
  {\prod_{i=1}^{\nt-1} \Gamma(\alpha_i)}
  \left(\prod_{i=1}^{\nt-1} {\pteam'_i}^{\alpha_i-1}\right)
  \delta\left(1-\sum_{i=1}^{\nt-1} \pteam'_i\right)
\end{equation}
which is a Dirichlet distribution with the same parameters
$\{\alpha_1,\ldots,\alpha_{\nt-1}\}$.

However, desideratum~\ref{d:reflection} is not satisfied, which we can
see explicitly by considering $\nt=3$ and assuming
$\alpha_1=\alpha_2=\alpha_3\equiv\alpha$, so that
\begin{equation}
  f_{\Ptvec}(\pteam_1,\pteam_2,\pteam_3|I_{\text{D}3})
  \propto (\pteam_1\pteam_2\pteam_3)^{\alpha-1}
  \,\delta(\pteam_1+\pteam_2+\pteam_3-1)
\end{equation}
or equivalently
\begin{equation}
  f_{\Ltvec}(\lteam_1,\lteam_2,\lteam_3|I_{\text{D}3})
  \propto e^{\alpha(\lteam_1+\lteam_2+\lteam_3)}
  \,\delta\left(e^{\lteam_1}+e^{\lteam_2}+e^{\lteam_3}-1\right)
\end{equation}
If we write
\begin{multline}
  \pteam_1+\pteam_2+\pteam_3 = (\pteam_1\pteam_2\pteam_3)^{1/3}
  \left(
    \left(\frac{\pteam_1}{\pteam_2}\frac{\pteam_1}{\pteam_3}\right)^{1/3}
    + \left(\frac{\pteam_2}{\pteam_1}\frac{\pteam_2}{\pteam_3}\right)^{1/3}
    + \left(\frac{\pteam_3}{\pteam_1}\frac{\pteam_3}{\pteam_2}\right)^{1/3}
  \right)
  \\
  = e^{\lteam_1}+e^{\lteam_2}+e^{\lteam_3}
  = e^{\frac{\lteam_1+\lteam_2+\lteam_3}{3}}
  \left(
    e^{\frac{\lpair_{12}+\lpair_{13}}{3}}
    + e^{\frac{-\lpair_{12}+\lpair_{23}}{3}}
    + e^{\frac{-\lpair_{13}-\lpair_{23}}{3}}
  \right)
\end{multline}
we can see
\begin{equation}
  \delta\left(e^{\lteam_1}+e^{\lteam_2}+e^{\lteam_3}-1\right)
  \propto
  \delta\left(\lteam_1+\lteam_2+\lteam_3
    +3\ln\left[
      e^{\frac{\lpair_{12}+\lpair_{13}}{3}}
      + e^{\frac{-\lpair_{12}+\lpair_{23}}{3}}
      + e^{\frac{-\lpair_{13}-\lpair_{23}}{3}}
    \right]
  \right)
\end{equation}
and so marginalizing over the combination $\lteam_1+\lteam_2+\lteam_3$
leaves a prior distribution
\begin{equation}
  f_{\Lpvec}(\lpvec|I_{\text{D}3})
  \propto
  \left(
    e^{\frac{\lpair_{12}+\lpair_{13}}{3}}
    + e^{\frac{-\lpair_{12}+\lpair_{23}}{3}}
    + e^{\frac{-\lpair_{13}-\lpair_{23}}{3}}
  \right)^{-3\alpha}
\end{equation}
We see that, for non-zero $\alpha$,
$f_{\Lpvec}(\lpvec|I_{\text{D}3}) \ne
f_{\Lpvec}(-\lpvec|I_{\text{D}3})$.  This is illustrated explicitly
for $\alpha=\frac{1}{2}$ in \fref{f:t3dir1}.  Another way of
expressing this asymmetry if we'd started with an ``anti-Dirichlet''
prior, i.e., requiring $\{1/\pteam_i\}$ to be Dirichlet distributed,
we'd have got the distribution
\begin{equation}
  f_{\Lpvec}(\lpvec|I_{\text{D}'3})
  \propto
  \left(
    e^{\frac{-\lpair_{12}-\lpair_{13}}{3}}
    + e^{\frac{\lpair_{12}-\lpair_{23}}{3}}
    + e^{\frac{\lpair_{13}+\lpair_{23}}{3}}
  \right)^{-3\alpha}
\end{equation}
\begin{figure}[t!]
  \centering
  \includegraphics[width=0.8\columnwidth]{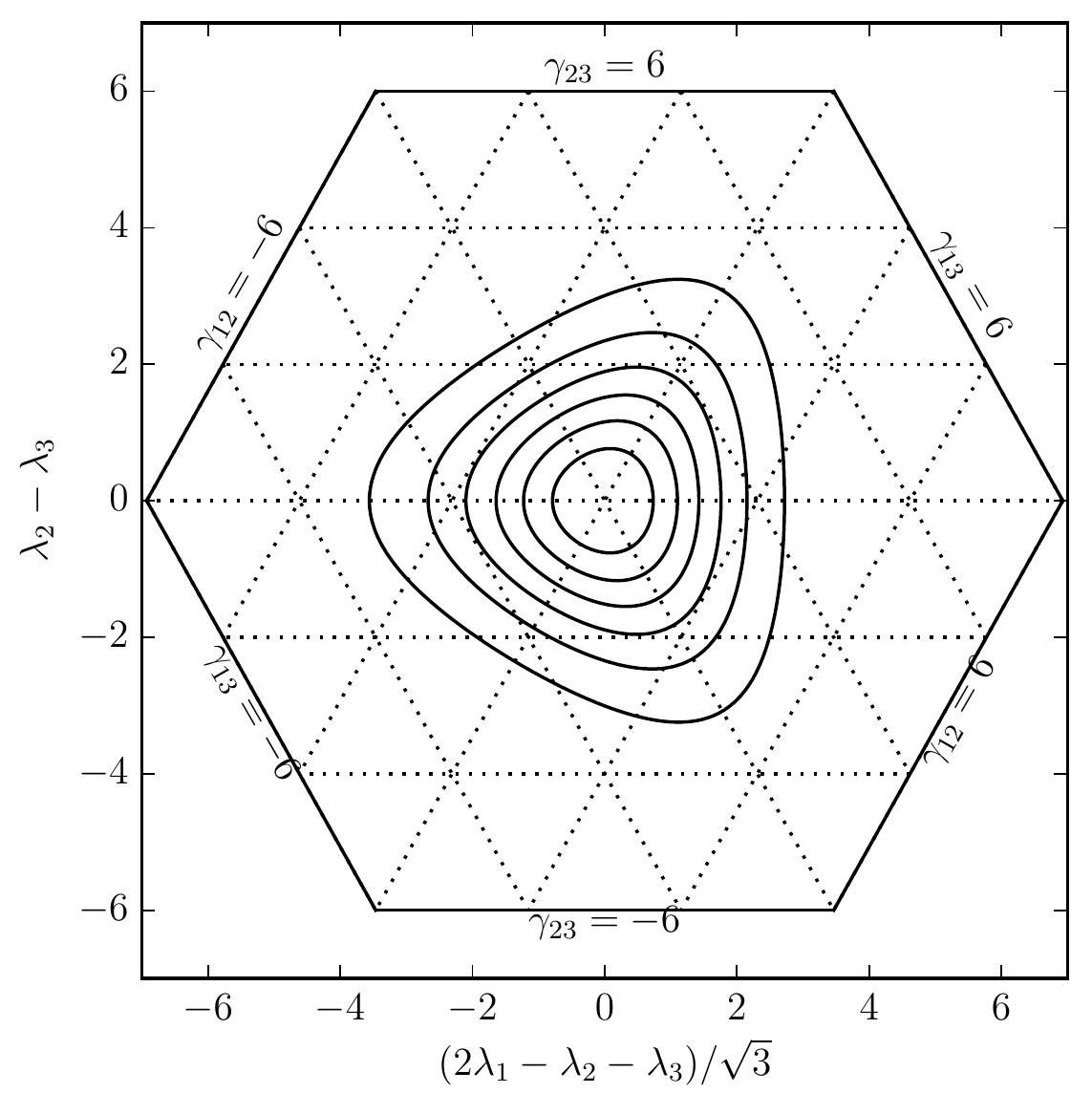}
  \caption{Contour plot of the prior probability distribution
    $f_{\Lpvec}(\lpvec|I_{\text{D}3})$ arising from the Dirichlet
    prior described in \sref{s:dirichlet} when $\alpha=1$ and $\nt=3$.
    We see that the distribution is not invariant under the inversion
    $\lpvec\rightarrow -\lpvec$, and therefore
    desideratum~\ref{d:reflection} is not satisfied.}
  \label{f:t3dir1}
\end{figure}

\subsubsection{Conjugate Prior Families}

\citet{DAVIDSON01121973} construct a conjugate prior family of the form
\begin{equation}
  f_{\Ptvec}(\ptvec|I_{\text{C}}) \propto
  \left(\prod_{i=1}^{\nt}\pteam_i^{\win_i^0}\right)
  \left(
    \prod_{i=1}^{\nt}\prod_{j=i+1}^{\nt}(\pteam_i+\pteam_j)^{-\nnum_{ij}^0}
  \right)
  \delta\left(1-\sum_{i=1}^{\nt} \pteam_i\right)
\end{equation}
where $\nnum_{ii}^0=0$.  In order to satisfy
desideratum~\ref{d:interchange} (interchange of teams), we require
that $\win_i^0=\win^0$ and $\nnum_{ij}^0=\nnum^0$ for all $i$ and
$j\ne i$, so the prior becomes
\begin{equation}
  f_{\Ptvec}(\ptvec|I_{\text{C}}) \propto
  \left(\prod_{i=1}^{\nt}\pteam_i\right)^{\win^0}
  \left(
    \prod_{i=1}^{\nt}\prod_{j=i+1}^{\nt}(\pteam_i+\pteam_j)
  \right)^{-\nnum^0}
  \delta\left(1-\sum_{i=1}^{\nt} \pteam_i\right)
  \ .
\end{equation}
Note that \citet{DAVIDSON01121973} motivate $\win_i^0$ and
$\nnum_{ij}^0$ as coming from a matrix $\wwin_{ij}^0$ (with
$\wwin_{ii}^0=0$) via $\win_i^0=\sum_{i=1}^{\nt}\wwin_{ij}^0$ and
$\nnum_{ij}^0=\wwin_{ij}^0+\wwin_{ji}^0$, which means that in
particular
$\sum_{i=1}^{\nt}\sum_{j=1}^{\nt}\nnum_{ij}^0=2\sum_{i=1}^{\nt}\win_i^0$,
which in the case of single $\nnum^0$ and $\win^0$ parameters would
require $\win^0=(\nt-1)\nnum^0/2$.  We will not impose that condition
at this stage, however.

To convert $f_{\Ptvec}(\ptvec|I_{\text{C}})$ into
$f_{\Ltvec}(\ltvec|I_{\text{C}})$, we note that
\begin{equation}
  \sum_{i=1}^{\nt}\pteam_i
  = \left(\prod_{k=1}^{\nt}\pteam_k\right)^{1/t}
  \sum_{i=1}^{\nt}\left(
    \prod_{j=1}^{\nt} \frac{\pteam_i}{\pteam_j}
  \right)^{1/t}
  = \exp\left(\frac{1}{\nt}\sum_{k=1}^{\nt}\lteam_k\right)
  \sum_{i=1}^{\nt}
  \exp\left(\frac{1}{\nt}\sum_{j=1}^{\nt}\lpair_{ij}\right)
\end{equation}
and since, for $\nt>0$,
\begin{equation}
  \delta\left(1-e^{u/t}\right)
  = \frac{\delta(u)}{\frac{1}{t}e^{u/t}}
  = t\delta(u)
  \ ,
\end{equation}
\begin{equation}
  \delta\left(1-\sum_{i=1}^{\nt} \pteam_i\right)
  =
  \nt\ \delta\left(
    \sum_{k=1}^{\nt}\lteam_k
    + \nt \ln \sum_{i=1}^{\nt}
    \exp\left(\frac{1}{\nt}\sum_{j=1}^{\nt}\lpair_{ij}\right)
  \right)
  \ .
\end{equation}
Similarly, we can write
\begin{equation}
  \begin{split}
    \prod_{i=1}^{\nt}\prod_{j=i+1}^{\nt}(\pteam_i+\pteam_j)
    &=
    \frac{1}{\prod_{k=1}^{\nt}2\pteam_k}
    \sqrt{
      \prod_{i=1}^{\nt}\prod_{j=1}^{\nt}(\pteam_i+\pteam_j)
    }
    \\
    &=
    \frac{1}{2^{\nt}\prod_{k=1}^{\nt}\pteam_k}
    \sqrt{
      \left(\prod_{k=1}^{\nt}\pteam_k\right)^{\nt}
      \left[
        \prod_{i=1}^{\nt}\prod_{j=1}^{\nt}
        \left(1+\frac{\pteam_i}{\pteam_j}\right)
      \right]
    }
    \\
    &= \frac{1}{2^{\nt}}
    \exp\left(\left(\frac{\nt}{2}-1\right)\sum_{k=1}^{\nt}\lteam_k\right)
    \left(
      \prod_{i=1}^{\nt}\prod_{j=1}^{\nt}
      \left[1+e^{\lpair_{ij}}\right]
    \right)^{1/2}
  \end{split}
\end{equation}
which makes the prior (recalling \eqref{e:fltpt})
\begin{equation}
  \begin{split}
    f_{\Ltvec}(\ltvec|I_{\text{C}})
    \propto&
    \exp\left(
      \left[\win^0+1-\nnum^0\left(\frac{\nt}{2}-1\right)\right]
      \sum_{k=1}^{\nt}\lteam_k
    \right)
    \left(
      \prod_{i=1}^{\nt}\prod_{j=1}^{\nt}
      \left[1+e^{\lpair_{ij}}\right]
    \right)^{-\nnum^0/2}
    \\
    &\times \delta\left(
      \sum_{k=1}^{\nt}\lteam_k
      + \nt \ln \sum_{i=1}^{\nt}
      \exp\left(\frac{1}{\nt}\sum_{j=1}^{\nt}\lpair_{ij}\right)
    \right)
  \end{split}
\end{equation}
When we marginalize over $\sum_{k=1}^{\nt}\lteam_k$, the Dirac delta
function sets $\exp\left(\sum_{k=1}^{\nt}\lteam_k\right)$ to
$\left(\sum_{i=1}^{\nt}
  \exp\left(\frac{1}{\nt}\sum_{j=1}^{\nt}\lpair_{ij}\right)\right)^{-\nt}$
and the prior becomes
\begin{equation}
  f_{\Lpvec}(\lpvec|I_{\text{C}})
  \propto
  \left[
    \sum_{i=1}^{\nt}
    \exp\left(\frac{1}{\nt}\sum_{j=1}^{\nt}\lpair_{ij}\right)
  \right]^{-\nt(\win^0+1-\nnum^0[\nt/2-1])}
  \left(
    \prod_{i=1}^{\nt}\prod_{j=1}^{\nt}
    \left[1+e^{\lpair_{ij}}\right]
  \right)^{-\nnum^0/2}
\end{equation}
The quantity in large square brackets is in general not symmetric
under the transformation $\lpvec\rightarrow -\lpvec$, while the
remainder of the expression is.  This means, to satisfy desideratum
\ref{d:reflection} (win-loss inversion), we should have
$\win^0=(\nt-2)\nnum^0/2-1$.  Note that this is not the same as the
condition $\win^0=(\nt-1)\nnum^0/2$ implied by
\citet{DAVIDSON01121973}'s conditions on $\win_i^0$ and
$\nnum_{ij}^0$.  The precise form of their restriction, however, comes
from the fact they wrote down a ``natural'' conjugate prior family for
$f_{\Ptvec}(\ptvec|I_{\text{C}})$; if they had started with
$f_{\Ltvec}(\ltvec|I_{\text{C}})$, they would have ended up with, in
the present notation, $\win^0+1=(\nt-1)\nnum^0/2$, which would also
not satisfy desideratum~\ref{d:reflection}.  However, if we hadn't
imposed the constraint $\sum_{i=1}^{\nt}\pteam_i=1$ (which is clearly
not invariant under $\pteam_i\rightarrow 1/\pteam_i$) in the first
place, the marginalization over $\sum_{k=1}^{\nt}\lteam_k$, would have
rendered $\win^0$ irrelevant and left us with
\begin{equation}
  \label{e:priorDS}
  f_{\Lpvec}(\lpvec|I_{\text{C}})
  \propto
  \left(
    \prod_{i=1}^{\nt}\prod_{j=1}^{\nt}
    \left[1+e^{\lpair_{ij}}\right]
  \right)^{-\nnum^0/2}
\end{equation}
in any event.  We therefore take \eqref{e:priorDS} as the form of
the conjugate prior $f_{\Lpvec}(\lpvec|I_{\text{C}})$, and note that
if $\nnum^0=2$, this reduces to the maximum entropy distribution
\eqref{e:priorme} [see also \eqref{e:meast}] considered in
\ref{s:maxent}.

We can show that \eqref{e:priorDS} violates desideratum~\ref{d:elim}
for any $\nnum^0\in(0,\infty)$ by considering the prior predictive
distribution and invoking Lemma~\ref{l:predelim}.  First, note that
for $\nt=2$, \eqref{e:priorDS} becomes
\begin{equation}
  \label{e:priorDS2}
  f_{\Lpvec}(\lpvec|I_{\text{C}2})
  \propto
  \left(
    \left[1+e^{\lpair_{12}}\right]
    \left[1+e^{\lpair_{21}}\right]
  \right)^{-\nnum^0/2}
\end{equation}
which is just the beta/generalized logistic prior \eqref{e:betalog}
with $\alpha=\beta=\nnum^0/2$.  Therefore, if we define $D$ to be a
set of results for which $\wwin_{12}=1=\wwin_{21}$, and let
$\nnum_{12}=2$ and $\nnum_{i3}=0$, \eqref{e:predbeta} implies
\begin{equation}
  p(D|\nnumvec,I_{\text{C}2})
  = \frac{\Gamma(\nnum^0)\Gamma(\frac{\nnum^0}{2}+1)\Gamma(\frac{\nnum^0}{2}+1)}
  {\Gamma(\frac{\nnum^0}{2})\Gamma(\frac{\nnum^0}{2})\Gamma(\nnum^0+2)}
  = \frac{1}{4(1+\frac{1}{\nnum^0})}
\end{equation}
We evaluate $p(D|\nnumvec,I_{\text{C}3})$ numerically for a range of
$\nnum^0$ values, plotted in \fref{f:predDS}, and find that for any
$0<\nnum^0<\infty$,
$p(D|\nnumvec,I_{\text{C}3})>p(D|\nnumvec,I_{\text{C}2})$
\begin{figure}
  \centering
  \includegraphics[width=0.8\columnwidth]{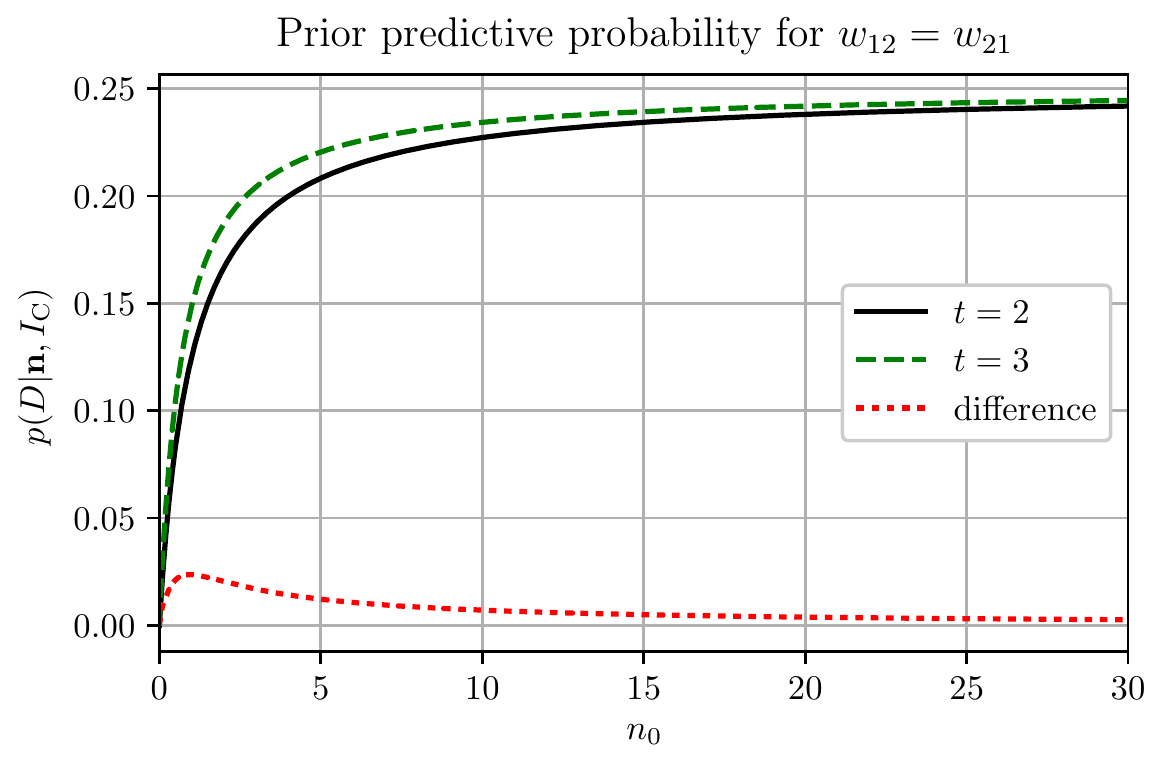}
  \caption{Prior predictive probability $p(D|\nnumvec,I_{\text{C}})$
    for the conjugate prior family of \citet{DAVIDSON01121973}, for a
    sequence of results in which $\wwin_{12}=\wwin_{21}=1$ and
    $\nnum_{12}=2$.  In order to satisfy
    desiderata~\ref{d:interchange} and~\ref{d:reflection}, we require
    $\nnum_{ij}^0=\nnum^0$ and $\win^0=(\nt-2)\nnum^0/2-1$, which
    produces the prior \eqref{e:priorDS}.  For any finite positive
    $\nnum^0$, the prior predictive probability from the three-team
    prior is larger than that of the two-team prior, indicating that
    desideratum~\ref{d:elim} is not satisfied.  (For $\nnum^0=0$ both
    versions reduce to the Haldane prior and
    $p(D|\nnumvec,I_{\text{C}})=0$.  As $\nnum^0\rightarrow\infty$,
    they become delta functions at $\lpvec=\zerovec$ and
    $p(D|\nnumvec,I_{\text{C}})\rightarrow 0.25$.)}
  \label{f:predDS}
\end{figure}

\subsubsection{Multivariate Gaussian Distribution}

\citet{Leonard1977} proposed a multivariate Gaussian prior on the
$\{\lteam_i\}$ of the form
\begin{equation}
  f_{\Ltvec}(\ltvec|I_{\text{G}}) \propto
  \exp\left(
    -\frac{1}{2}\sum_{i=1}^{\nt}\sum_{j=1}^{\nt}
    (\lteam_i-\mu_i)[\sigma^{-2}]_{ij}(\lteam_j-\mu_j)
  \right)
\end{equation}
where $\{[\sigma^{-2}]_{ij}\}$ are the elements of the inverse of a
positive semi-definite covariance matrix with elements
$\{[\sigma^2]_{ij}\}$.

Desideratum~\ref{d:proper} will be satisfied if the covariance matrix
$\{[\sigma^2]_{ij}\}$ is positive definite, so that the prior is
normalizable.

Desideratum~\ref{d:interchange} requires that all of the $\{\mu_i\}$
have the same value $\mu$, all of the variances $\{[\sigma^2]_{ii}\}$
have the same value $\sigma^2$, and all of the cross-covariances
$\{[\sigma^2]_{ij}|i\ne j\}$ have the same value $\rho\sigma^2$.  In
order for the matrix $\{[\sigma^2]_{ij}\}$ to be positive definite we
must have $\sigma^2>0$ and $-\frac{1}{\nt-1}<\rho<1$.

These conditions guarantee that desideratum~\ref{d:reflection} is
satisfied.

Since the distribution $f_{\Lpvec}(\lpvec|I_{\text{G}})$ is unchanged
by the transformation $\lambda_i\rightarrow\lambda_i-\mu$, we can
assume without loss of generality that $\mu=0$.  We are thus left with
$\rho$ and $\sigma^2$ as the adjustable parameters of the
distribution.  However, if we make the transformation
$\lambda_i\rightarrow\lambda_i+a\sum_{j=1}^{\nt}\lambda_j$, which
leaves $f_{\Lpvec}(\lpvec|I_{\text{G}})$ unchanged,
$f_{\Ltvec}(\ltvec|I_{\text{G}})$ becomes a multivariate Gaussian with
covariance matrix
\begin{equation}
  \sigma^2
  \left[
    \rho+\delta_{ij}(1-\rho)
    + a(2+a\nt)(\nt+1-\rho)
  \right]
\end{equation}
if $a=-\frac{1}{t}\left(1+\sqrt{\frac{1-\rho}{\nt\rho+1-\rho}}\right)$
or $a=-\frac{1}{t}\left(1-\sqrt{\frac{1-\rho}{\nt\rho+1-\rho}}\right)$
the covariance matrix becomes diagonal, with a variance equal to
$(1-\rho)\sigma^2$.  Either value for $a$ is guaranteed to be real by
the conditions on $\rho$ which ensure a positive definite correlation
matrix.  Thus $f_{\Lpvec}(\lpvec|I_{\text{G}})$ is equivalent to a
product of independent Gaussian distributions for each $\lteam_i$.
For simplicity, we refer to the variance of each of these
distributions as $\sigma^2$ rather than $(1-\rho)\sigma^2$.

Since the prior $f_{\Lpvec}(\lpvec|I_{\text{G}})$ is equivalent to
independent distributions on the $\{\lambda_i\}$, it is invariant
under elimination of teams, and satisfies desideratum~\ref{d:elim}.

\subsubsection{Separable Priors}

Thus far, the only prior considered to satisfy all four of our
desiderata is the Gaussian prior which is equivalent to
\begin{equation}
  f_{\Ltvec}(\ltvec|I_{\text{G}})
  = \frac{1}{(2\pi\sigma^2)^{\nt/2}}
  \exp
  \left(
    -\sum_{i=1}^{\nt}\frac{\lteam_i^2}{2\sigma^2}
  \right)
\end{equation}
This is an example of a separable prior of the form
\begin{equation}
  f_{\Ltvec}(\ltvec|I) = \prod_{i=1}^{\nt}f_{\Lteam_i}(\ltvec_i|I)
\end{equation}
A prior of this form, which assigns prior pdfs to the strengths (or
log-strengths) of the teams is guaranteed by its construction to
satisfy desideratum~\ref{d:elim} (invariance under team-elimination).
It will also satisfy desideratum~\ref{d:interchange} (interchange) if
the distributions for the different $\{\ltvec_i\}$ are identical
($f_{\Lteam_i}(\ltvec_i|I)=f_{\Lteam}(\ltvec_i|I)$),
desideratum~\ref{d:reflection} (win-loss interchange) if the
distribution $f_{\Lteam}(\ltvec_i|I)$ is even
($f_{\Lteam}(-\ltvec_i|I)=f_{\Lteam}(\ltvec_i|I)$) and
desideratum~\ref{d:proper} (normalizable) if $f_{\Lteam}(-\ltvec|I)$
is a proper prior.

\subsubsection{Beta-Separable Priors}

One family of separable priors can be constructed by defining
\begin{equation}
  \zeta_i=\frac{\pteam_i}{1+\pteam_i}=(1+e^{-\lteam_i})^{-1}
\end{equation}
and assuming that $\zeta_i$ obeys a Beta distribution.
Explicitly,
\begin{equation}
  f_{\Zeta_i}(\zeta_i|I_{\text{B}})
  \propto \zeta_i^{\alpha_i-1}(1-\zeta_i)^{\beta_i-1}
\end{equation}
Since $\zeta_i=\text{logistic}(\lteam_i)$, just as
$\ppair_{ij}=\text{logistic}(\lpair_{ij})$, the prior on $\lteam_i$ is
a generalized logistic distribution of Type IV
\citep{Prentice:1976,Nassar:2012}:
\begin{equation}
  f_{\Lteam_i}(\lteam_i|I_{\text{B}}) \propto
  (1+e^{-\lteam_i})^{-\alpha_i}(1+e^{\lteam_i})^{-\beta_i}
\end{equation}
To enforce desideratum~\ref{d:reflection}, we require
$\alpha_i=\beta_i$, and for desideratum~\ref{d:interchange}, we
require that all $\alpha_i$ and $\beta_i$ parameters are the same,
$\alpha_i=\beta_i=\eta$, which makes the prior a generalized logistic
distribution of Type III.
\begin{equation}
  f_{\Lteam_i}(\lteam_i|I_{\text{B}}) \propto
  (1+e^{-\lteam_i})^{-\eta}(1+e^{\lteam_i})^{-\eta}
\end{equation}
An appealing feature of Beta-separable prior is that its functional
form is similar to the likelihood
\begin{equation}
  p(D|\ltvec)
  = \prod_{i=1}^{\nt} \prod_{j=i+1}^{\nt}
  (1+e^{-\lpair_{ij}})^{-\wwin_{ij}}(1+e^{\lpair_{ij}})^{\nnum_{ij}-\wwin_{ij}}
\end{equation}
In particular, the posterior is proportional to the likelihood
function which would arise by adding to the actual game results a set
of ``fictitious games'' corresponding to $\eta$ wins and $\eta$ losses
for each actual team against a fictitious team assumed to have a
strength of $1$.  So any method for obtaining maximum likelihood
estimates such as that of \citet{Ford:1957} could be adapted to
obtaining maximum a priori estimates with this prior.  This method,
with $\eta=\frac{1}{2}$, has been used by \citet{KRACH1993} to ensure
regularity of the estimates of Bradley-Terry strengths.  Another
``obvious'' choice is $\eta=1$, which is equivalent to a uniform prior
on $\zeta_i$.  (The Haldane prior is $\eta=0$.)

\section{Conclusions}

We have considered various families of prior distributions for team
strengths in the Bradley-Terry model.  Motivated by the application of
a Bayesian Bradley-Terry model to rate teams based only on their game
results, we have evaluated these priors according to the desiderata of
invariance under interchange of teams (\ref{d:interchange}),
interchange of winning and losing (\ref{d:reflection}) and elimination
of irrelevant teams from the model (\ref{d:elim}), as well as
normalizability (\ref{d:proper}).  A Haldane-like prior of complete
ignorance is not normalizable (violation of \ref{d:proper}), although
it satisfies the other desiderata.  A prior based on maximum entropy
arguments, as well as one from the conjugate family of
\citet{DAVIDSON01121973} which is required to obey the other
desiderata, will depend on the number of teams for which it was
constructed (violation of \ref{d:elim}).  The same is true for the
Jeffreys prior.  A Dirichlet prior on the team strengths
\citep{Chen19849} can be made to satisfy the other desiderata, but
will not be invariant under interchange of the definitions of winning
and losing (violation of \ref{d:reflection}).  Distributions can be
constructed which satisfy the desiderata by imposing independent
priors on the strengths of all of the teams.  In particular, a
multivariate Gaussian in the log-strengths \citep{Leonard1977} which
satisfies the desiderata is equivalent to identical independent
Gaussian priors on each of the log-strengths.  Another simple family
of separable prior distributions imposes independent generalized
logistic distributions on the log-strengths.  In each of these last
two cases, a single parameter remains.  \citet{Phelan2017} consider
the relationship between the two, and propose a hierarchical method to
estimate these parameters rather than assuming values for them.

\section*{Acknowledgments}

The author wishes to thank Kenneth Butler and Gabriel Phelan for
useful discussions.

\end{document}